\documentclass[12pt]{article}
\usepackage{amsmath,amssymb,amsthm}
\usepackage{graphics}

\usepackage{natbib}

\pdfoutput=1

\newcommand{\x}{\mathbf{x}}
\newcommand{\n}{\mathbf{n}}
\newcommand{\E}{\mathbf{E}}
\newcommand{\J}{\mathbf{J}}

\newcommand{\Et}[1]{\mathbf{E}_{t}^{#1}}
\newcommand{\Jt}[1]{\mathbf{J}_{t}^{#1}}
\newcommand{\En}[1]{\mathbf{E}_{\mathbf{n}}^{#1}}
\newcommand{\Jn}[1]{\mathbf{J}_{\mathbf{n}}^{#1}}

\newcommand{\bfm}[1]{\mbox{\boldmath ${#1}$}}
\newcommand{\BGs}{\bfm\sigma}

\newtheorem{thm}{Theorem}[section]
\newtheorem{lemma}{Lemma}[section]

\numberwithin{equation}{section}

\begin{document}
\nocite{*}
\title{Effective Conductivities of Thin-Interphase Composites}
\author{Bacim Alali$^{\tiny a}$ and Graeme W. Milton$^{\tiny b}$\\
\
\footnotesize{$^{\tiny a}$ Department of Scientific Computing, Florida State University, Tallahassee, FL}\\
\footnotesize{$^{\tiny b}$ Mathematics Department, University of Utah, Salt Lake City, UT}}
\date{}
\maketitle

%
%
%
%


\begin{abstract}
A method is presented for approximating the effective conductivity of composite media with thin interphase regions, 
which is exact to first order in the interphase thickness. 
The approximations are computationally efficient in the sense the fields need to be computed
only in a reference composite in which the interphases have been replaced by perfect interfaces.
The results apply whether any two phases of the composite are separated by a single interphase or multiple interphases, whether 
the conductivities
of the composite phases are isotropic or anisotropic,  and whether the thickness of an interphase is uniform or varies as a 
function of position. It is assumed that the conductivities of the interphase materials have intermediate values as opposed to 
very high or very low conductivities.

\end{abstract}





\section{Introduction}
\label{intro}

Composite media with thin interphases between adjacent phases are used widely in engineering, biomedical, and technological applications. 
Examples include
coated fiber or coated particle reinforced materials and composites with adhesively bonded joints, see for instance
\cite{kim1998engineered,banea2009adhesively,ramakrishna2001biomedical}.
In this work we consider the problem of computing the effective conductivity in a composite with  thin interphase regions
between some of the composite's constituents, see Figure~\ref{fig:interphases}. The fact that the composite has a thin interphase (or interphases) makes
 the numerical computations of the fields and the effective properties a difficult task. 
Therefore, approximate models of thin interphase 
composites that do not require solving for the fields inside the interphase regions are desirable, \cite{benveniste2012two}.
 
\begin{figure}[t]

\centering
\scalebox{0.115}{\includegraphics{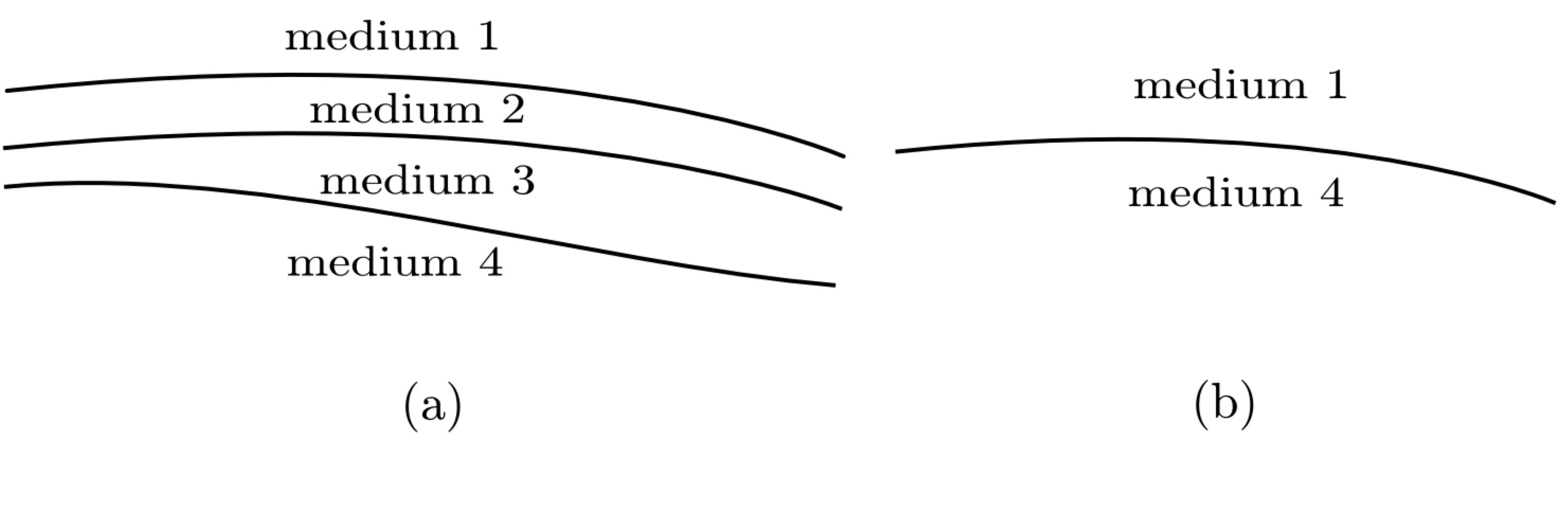}}
\caption{(a) A composite in which media 2 and 3 form an interphase region between media 1 and 4. (b) A reference composite in which the interphases
of the composite in (a) have been replaced by an interface.}
\label{fig:interphases}
\end{figure}

There has been an extensive study of interface models, in which an interphase region between 
two media is replaced by an interface, allowing direct contact of the media, along with appropriate interface conditions.
This interface is usually referred to as an imperfect interface when either the potential or the normal flux component 
has a jump discontinuity across the interface. The effective conductivity in composites with highly conducting
or poorly conducting imperfect interfaces has been studied by
\cite{Benveniste19861537,benveniste1987effective,torquato1995effect,lipton1996composites,lipton1997variational, miloh1999effective,
lipton2001bounds, le2010size}, among others. Numerical methods for computing the effective conductivity in composites with highly conducting
or poorly conducting imperfect interfaces have been recently developed by \cite{yvonnet2008numerical, yvonnet2011general}.

Although the majority of studies in the literature of thin interphases focus on the case when there is a high contrast 
between the conductivities of the interphase and its adjacent media, there have been some works that deal with intermediate 
values of conductivities without the high-contrast assumption. Notable examples of such studies include the work of 
 \cite{hashin2001thin} and the works of \cite{benveniste2006general,benveniste2006hn,benveniste2007interface,benveniste2012two}.
In the work of \cite{hashin2001thin} an interface model is introduced, in which the interphase's conductivity is arbitrary
ranging from zero to infinity, and the effect of the interphase on the effective conductivity is discussed. However, this work 
is restricted to having isotropic and homogeneous phases and the thickness of the interphase is assumed to be constant. 
In the above studies by Benveniste et. al. two models of thin interphase materials are developed of which  one is an interface model
with the interphase replaced by an interface. In the second model the geometry of the thin interphase is left intact and 
is characterized by conditions
in which the fields are evaluated in the adjacent media at both sides of the interphase and do not involve the
fields within the interphase.
Some of the interesting features of these models are that the interphase can have anisotropic and inhomogeneous conductivities.
However, it is assumed that the thickness of the interphase region is constant. It is also noted that
these studies do not discuss the effect of the thin interphase on the effective conductivity.
For a more comprehensive review of the topic (on the modeling of thin interphases), we refer the reader to the works given in
\cite{benveniste2006general,Gu20111413}.




In this work we present a novel method that provides an approximation to the effective conductivity
and combines different desirable features. 
These include that the method is exact to first order in the interphase thickness and that 
the method is computationally inexpensive in the sense that it does not involve computing the fields within the interphase region.
Additionally, the method applies for homogeneous as well as
inhomogeneous interphases, that is whether a single interphase 
separates two materials or multiple (possibly infinitely many) interphases. 
Moreover, the conductivities
of the composite components, including the interphases, can be isotropic or anisotropic. 
Furthermore, the geometry of the interphase
is arbitrary and its thickness $h$ is not assumed to be constant but can vary slowly, on a length scale
large compared to $h$. 
However, the total
thickness of the interphase region must be sufficiently small in order for our approximation of the effective conductivity to hold.
It is also assumed that the interphase conductivities have intermediate values as opposed to 
very high or very low conductivities. Numerical results for a particular example of a thin-interphase composite show
that our approximation agrees very well with the exact effective conductivity when
the interphase is thin and has intermediate values of conductivity. 

Our method is based on estimating the change in the effective conductivity as the total thickness of the interphases vanishes. Thus 
we  compare the effective conductivity of the composite to a reference composite in which the interphase materials in the original 
composite  have been removed and replaced by a perfect interface, with the usual conditions of continuity of potential and continuity of flux
across this interface . Computing the effective conductivity and the fields inside the reference
composite is an easier problem that may, for example, be solved using integral equations. We refer to this as the unperturbed problem.
The method is based on the following estimate, which holds when the total thickness of the interphases is sufficiently small,
\begin{equation}
 \label{segma_effctive_appx}
\BGs^*\approx\widetilde{\BGs}^*= \BGs^*_0 + \delta\BGs^*.
\end{equation}
Here $\BGs^*$ is the effective conductivity of the composite with thin interphases, 
$\widetilde{\BGs}^*$ is our approximation,  
$\BGs^*_0$ is the 
effective conductivity of the identical composite but with the interphase materials removed and replaced by one material of 
the two neighboring phases, and  $\delta \BGs^*$ is the change in the 
effective conductivity due to inserting the interphase materials.
The formulas that we
present and derive in Section~\ref{sec:interphase} show that  in order to compute $\delta\BGs^*$ we need to compute the fields 
only at material interfaces inside the unperturbed problem (or reference composite).

We emphasize that the method introduced here is quite general and can easily be extended to elasticity, piezoelectricity or other 
coupled field problems. Indeed its basis is the interface-shift formula of \cite{Milton:2002:TOC} which was developed 
in this more general context. We have chosen to focus on the conductivity case 
for simplicity and to make the presentation less abstract. 

This article is organized as follows. In Section~\ref{sec:rev}, a review is given for computing the change  in the effective conductivity due to translating an interface.
In Section~\ref{sec:interphase}, a formula for computing the change  in the effective conductivity due to inserting an interphase
region is derived. Finally, Section~\ref{sec:numerical} provides numerical results for  the effective conductivity
of the doubly coated sphere assemblage and shows a comparison between  our approximation, the exact result, and two other 
known approximations.

\section{Review}
\label{sec:rev}
In this section we review how the effective conductivity of a composite changes due to a shift of a phase boundary. 
The review here is based on the treatment given in Section 16.6 of \cite{Milton:2002:TOC}. For completeness of the presentation, we 
will include a derivation of Milton's interface-shift formula (\ref{interface_formula1}), in the case of the conductivity of composites.

Consider a periodic composite with a period cell given by $\Omega\subset\mathbb{R}^d$. Let $\Gamma=\Gamma_\eta$ be an 
interface between two phases, where $\eta$ is a continuous parameter such that the position of the interface $\Gamma_\eta$ 
changes when $\eta$ changes. We denote by $\BGs_\eta$ the conductivity inside the composite and assume that $\BGs_\eta(\x)=\BGs^i$ 
for $\x$ inside the $i$-th phase, where $\BGs^i$ is a 
symmetric
matrix representing an anisotropic 
conductivity. 
The electric 
and current fields are denoted by $\E_\eta$ and $\J_\eta$, respectively, and satisfy the differential constraints,
\begin{eqnarray}
\label{pde}
\nabla\cdot\J_\eta=0, \;\;\; \E_\eta=\nabla u_\eta,\;\;\; \J_\eta=\BGs_\eta\E_\eta,\;\;\;\langle \E_\eta\rangle=\E_0,
\end{eqnarray}
where $u_\eta$ is the electric potential and $\E_0$ is the applied electric field which is chosen to be independent of $\eta$. Here, and in the rest of this article, the brackets $\langle\cdot\rangle$ denote taking the average over $\Omega$, that is,
\[
\langle F \rangle=\frac{1}{|\Omega|} \int_\Omega F.
\] 
We assume that the potential $u_\eta$ and the flux $\n\cdot\J_\eta$ are continuous across the interface $\Gamma_\eta$, where $\n$ is the normal to the interface. These continuity conditions can be explicitly written as
\begin{eqnarray}
\label{interface_conds}
u_\eta^+(\x)=u_\eta^-(\x), \;\;\; \n(\x)\cdot\J_\eta^+(\x)=\n(\x)\cdot\J_\eta^-(\x),\;\;\;\mbox{ for all } \x\in\Gamma_\eta,
\end{eqnarray}
where the plus and minus superscripts denote the two sides of the interface with the convention that the normal is directed outward from the $+$ side pointing toward the $-$ side, see Figure~\ref{fig:main}.

\begin{figure}[t]

\centering
\scalebox{0.065}{\includegraphics{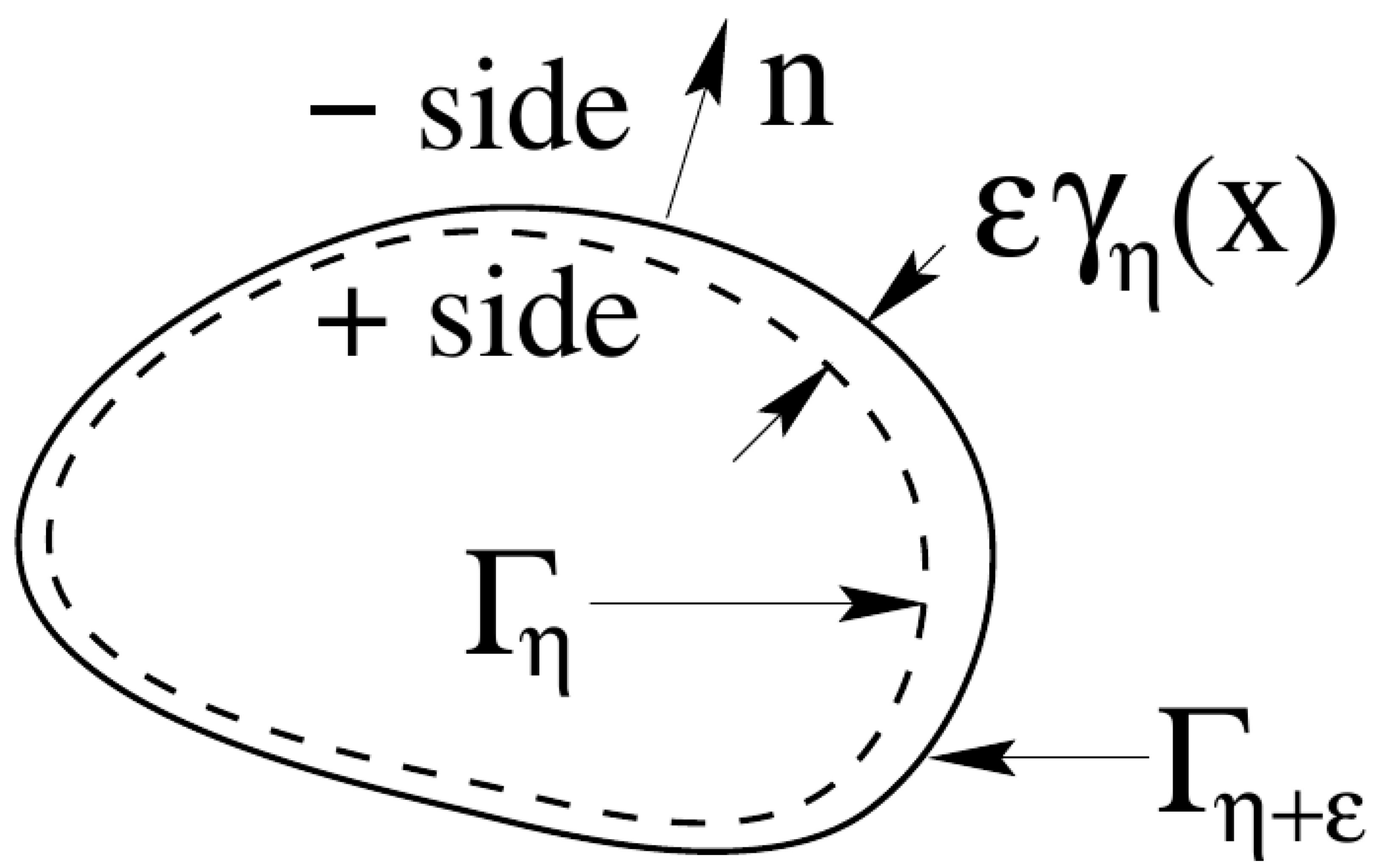}} 
\caption{Perturbation of an interface separating two phases. The distance between the interfaces $\Gamma_\eta$ and $\Gamma_{\eta+\epsilon}$
is approximated, to the first order in $\epsilon$, by $\epsilon |\gamma_\eta(\x)|$.}
\label{fig:main}
\end{figure}

For a small perturbation in $\eta$, $ \eta\rightarrow \eta+\epsilon$, the translated interface $\Gamma_{\eta+\epsilon}$ can be approximated to the first order in $\epsilon$ by
\begin{eqnarray}
\label{Gamma_appx}
\Gamma'_{\eta,\epsilon}=\{\x'|\x'=\x+\epsilon \gamma_\eta(\x) \n(\x),\;\; \x\in\Gamma_\eta\},
\end{eqnarray}
where $\gamma_\eta$ is a scalar such that the quantity $|\epsilon\gamma_\eta(\x)|$ represents the distance traveled by the point $\x\in\Gamma_\eta$ along the direction of $\n(\x)$. We note here that the fact that the potential $u_{\eta+\epsilon}$ is continuous across $\Gamma_{\eta+\epsilon}$ implies that, to the first order in $\epsilon$, 
\begin{eqnarray}
\label{u_at_Gamma_appx}
u_{\eta+\epsilon}^+(\x+\epsilon\gamma_\eta(\x)\n(\x))\approx u_{\eta+\epsilon}^-(\x+\epsilon\gamma_\eta(\x)\n(\x)).
\end{eqnarray}

Let $\BGs^*_\eta$ be the effective conductivity of the composite. Then the derivative of the energy with respect to $\eta$ is given by 
\begin{eqnarray}
\label{d_deta}
\nonumber
\frac{d}{d\eta}\left(\E_0\cdot\BGs^*_\eta\E_0\right) &=& 
\frac{d}{d\eta}\langle\E_\eta\cdot\BGs_\eta\E_\eta\rangle \\
\nonumber
&=& \frac{2}{|\Omega|}\int_{\Omega \setminus\ \Gamma_\eta}\frac{d \E_\eta}{d\eta}\cdot\BGs_\eta\E_\eta\\
& & +
\frac{1}{|\Omega|}\int_{\Gamma_\eta}\gamma_\eta
[\E_\eta^{+}\cdot\BGs^{+}\E_\eta^{+}- \E_\eta^{-}\cdot\BGs^{-}\E_\eta^{-}],
\end{eqnarray}
where the first integral represents the change in the energy due to variations in the field $\displaystyle\frac{d \E_\eta}{d\eta}$, treating the interface as fixed, and the second integral represents the change in the energy due to translating the interface, treating the field $\E_\eta$ as fixed outside the neighborhood of the interface.

Integrating by parts, the first integral in (\ref{d_deta}) becomes
\begin{eqnarray}
\label{d_deta_core_integ}
\nonumber
\int_{\Omega \setminus\ \Gamma_\eta}\frac{d \E_\eta}{d\eta}\cdot\BGs_\eta\E_\eta &=& \int_{\Gamma_\eta} \frac{\partial u^{+}_{\eta}}{\partial \eta} (\n\cdot\J^+_\eta)-\frac{\partial u^{-}_{\eta}}{\partial \eta} (\n\cdot\J^-_\eta)\\
&=& \int_{\Gamma_\eta} \left(\frac{\partial u^+_\eta}{\partial \eta}-\frac{\partial u^-_\eta}{\partial \eta}\right)(\n\cdot\J^+_\eta),
\end{eqnarray}
where in the last equality we used the continuity of the flux at the interface which is given by (\ref{interface_conds}). 
By expanding (\ref{u_at_Gamma_appx}) and equating the terms of order $\epsilon$, we obtain
\begin{eqnarray}
\label{gradu_at_Gamma_appx}
\frac{\partial u^+_\eta}{\partial \eta}+\gamma_\eta \n\cdot\nabla u^+_\eta= \frac{\partial u^-_\eta}{\partial \eta}+\gamma_\eta \n\cdot\nabla u^-_\eta.
\end{eqnarray}
By combining (\ref{gradu_at_Gamma_appx}) and (\ref{d_deta_core_integ}), we see that
\begin{eqnarray}
\label{d_deta_core_integ2}
\nonumber
\int_{\Omega \setminus\ \Gamma_\eta}\frac{d \E_\eta}{d\eta}\cdot\BGs_\eta\E_\eta &=& \int_{\Gamma_\eta} \gamma_\eta  \n\cdot\left(\nabla u^+_\eta- \nabla u^-_\eta\right) (\n\cdot\J^+_\eta)\\
&=& \int_{\Gamma_\eta} \gamma_\eta   \left(\nabla u^+_\eta- \nabla u^-_\eta\right)\cdot (\n\otimes\n) \J^+_\eta.
\end{eqnarray}

Let $w_1(\x),w_2(\x),\ldots,w_{d-1}(\x)$, and $\n(\x)$ be an orthonormal basis at each point $\x\in\Gamma_\eta$. Then since the potential is continuous across the interface, it follows that the tangential derivative of the potential is also continuous across  the interface and hence
\begin{eqnarray}
\nonumber
w_i\cdot\left(\nabla u^+_\eta- \nabla u^-_\eta\right)=0, \;\; \mbox{ for } i=1,2,\ldots,d-1.
\end{eqnarray}
Using this fact, (\ref{d_deta_core_integ2}) can be written as
\begin{eqnarray}
\label{d_deta_core_integ3}
\nonumber
\int_{\Omega \setminus\ \Gamma_\eta}\frac{d \E_\eta}{d\eta}\cdot\BGs_\eta\E_\eta &=& \int_{\Gamma_\eta} \gamma_\eta   \left(\nabla u^+_\eta- \nabla u^-_\eta\right)\cdot (\sum_{i=1}^{d-1}w_i\otimes w_i+\n\otimes\n) \J^+_\eta\\
\nonumber
&=&\int_{\Gamma_\eta} \gamma_\eta   \left(\nabla u^+_\eta- \nabla u^-_\eta\right)\cdot \J^+_\eta\\
&=&\int_{\Gamma_\eta} \gamma_\eta   \left(\E^-_\eta-\E^+_\eta\right)\cdot \BGs^+\E^+_\eta.
\end{eqnarray}
Using the same arguments above, only replacing $\n\cdot\J^+_\eta$  by $\n\cdot\J^-_\eta$ in (\ref{d_deta_core_integ}), gives
\begin{eqnarray}
\label{d_deta_core_integ4}
\int_{\Omega \setminus\ \Gamma_\eta}\frac{d \E_\eta}{d\eta}\cdot\BGs_\eta\E_\eta =\int_{\Gamma_\eta} \gamma_\eta   \left(\E^-_\eta-\E^+_\eta\right)\cdot \BGs^-\E^-_\eta.
\end{eqnarray}
By substituting (\ref{d_deta_core_integ3}) and (\ref{d_deta_core_integ4}) back into (\ref{d_deta}), we obtain
\begin{eqnarray}
\label{d_deta2}
\nonumber
\E_0\cdot\frac{d \BGs^*}{d\eta}\E_0&=&\frac{d}{d\eta}\left(\E_0\cdot\BGs^*_\eta\E_0\right) \\
\nonumber
&=& 
\frac{1}{|\Omega|}\int_{\Gamma_\eta}\gamma_\eta
[\E_\eta^{-}\cdot\BGs^{+}\E_\eta^{+}- \E_\eta^{+}\cdot\BGs^{-}\E_\eta^{-}]\\
&=& \frac{1}{|\Omega|}\int_{\Gamma_\eta}\gamma_\eta
[\E_\eta^{-}\cdot\J_\eta^{+}- \E_\eta^{+}\cdot\J_\eta^{-}].
\end{eqnarray}

The following theorem summarizes the results of this section.
\begin{thm}[\cite{Milton:2002:TOC}]
\label{thm:main}
Assume that an interface $\Gamma$ is translated by a distance $\epsilon \gamma(\cdot)$ in the direction normal to the interface, keeping the applied electric field $\langle \E\rangle$ fixed. Then the resultant change in the energy, to the first order in $\epsilon$, is given by
\begin{equation}
\label{interface_formula1}
\langle \E\rangle\cdot\delta\BGs^*\langle \E\rangle\approx 
\frac{1}{|\Omega|}\int_{\Gamma} \epsilon \gamma(\x)
\left[\E^{-}(\x)\cdot\J^{+}(\x)-\E^{+}(\x)\cdot\J^{-}(\x)\right]\;d\x,
\end{equation}
where $\delta\BGs^*$ denotes the change in the effective conductivity $\BGs^*$.
\end{thm}


\section{Change in the effective conductivity of thin-interphase composites}
\label{sec:interphase}

Consider a multi-phase composite consisting of a finite number of homogeneous constituents. Let $\BGs^0$ and $\BGs^1$ be the anisotropic conductivities  of two of these constituents such that between them lies a set of $K-1$ thin interphases, whose total thickness is $h$, having anisotropic conductivities given by $\BGs^2,\BGs^3, \ldots, \BGs^K$. We analyze the  change in the effective conductivity of this composite as the interphases total thickness vanishes ($h\rightarrow 0$). In the limit, the interphases are replaced by an interface between the $0$-phase and the $1$-phase of this composite.

In Sections \ref{sec:single} and \ref{sec:double} we present and derive formulas for this change in the effective conductivity for the case of a single thin interphase and for the case of two   thin interphases, respectively. The  general case of a finite set of thin interphases is discussed in Subsection \ref{sec:multiple}. Section \ref{sec:graded} focuses on the case when the interphase conductivity is continuously varying.

\subsection{Preliminaries}
\label{sec:pre}
From Theorem \ref{thm:main}, translating an interface $\Gamma$ by a distance $h$ in the direction normal to the interface, keeping $\langle \E\rangle$ fixed, results in a change in the energy, to the first order in $h$, given by
\begin{equation}
\label{interface_formula2}
\langle \E\rangle\cdot\delta\BGs^*\langle \E\rangle\approx 
\frac{1}{|\Omega|}\int_{\Gamma} h 
\left[\E^{-}\cdot\J^{+}-\E^{+}\cdot\J^{-}\right].
\end{equation} 
Note that the integration variable in (\ref{interface_formula2}) has been suppressed for conciseness. This result can be generalized to include the case of simultaneously shifting more than one interface. Specifically, if a finite number of interfaces $\Gamma_1,\ldots,\Gamma_N$ are shifted by distances $t_1,\ldots,t_N$, in the directions $\n_1,\ldots,\n_N$, respectively, where $\n_i$ is the direction normal to interface $\Gamma_i$, then the resultant change in the energy is given by
\begin{equation}
\label{interface_formula3}
\langle \E\rangle\cdot\delta\BGs^*\langle \E\rangle\approx 
\sum_{i=1}^{N}\frac{1}{|\Omega|}\int_{\Gamma_i} t_i 
\left[\E_i^{-}\cdot\J_i^{+}-\E_i^{+}\cdot\J_i^{-}\right].
\end{equation}

For a field $\mathbf F$, we denote by $\mathbf F_t$ and $\mathbf F_n$ the tangential and normal components of the field, respectively. 
Using this notation, we present the following lemma which provides an equivalent yet useful way for writing (\ref{interface_formula2}).
\begin{lemma}
\label{lem:1}
Let $\BGs^+$ and $\BGs^-$ be the conductivities of phase $+$ and phase $-$ as described above. Then
the change in the effective conductivity is given by
\begin{eqnarray}
\label{interface_formula2_equiv1}
\langle \E\rangle\cdot\delta\BGs^*\langle \E\rangle&\approx&\frac{1}{|\Omega|}\int_{\Gamma} h 
\left[\E^{-}\cdot\J^{+}-\E^{+}\cdot\J^{-}\right]\\
\nonumber
&=&\frac{1}{|\Omega|}\int_{\Gamma} h \left[
(\BGs^+-\BGs^-) \Et{+}\cdot\Et{+}-\left((\BGs^+)^{-1}-(\BGs^-)^{-1}\right)\Jn{+}\cdot \Jn{+}\right].\\
\label{interface_formula2_equiv2}
\end{eqnarray}
\end{lemma}
\begin{proof}
We expand the integrand by writing the fields in terms of their tangential and normal components to obtain
\begin{eqnarray}
\nonumber
\E^{-}\cdot\J^{+}&=& (\Et{-}+\En{-})\cdot(\Jt{+}+\Jn{+})\\
&=&\Et{-}\cdot\BGs^{+}\Et{+}+\Et{-}\cdot\Jn{+}+\En{-}\cdot\Jt{+}+\En{-}\cdot \Jn{+}.
\label{interface_formula2_equiv3}
\end{eqnarray}
Using the facts that $\Et{-}=\Et{+}, \Jn{-}=\Jn{+}$, and $\En{-}=(\BGs^{-})^{-1}\Jn{-}$, (\ref{interface_formula2_equiv3}) becomes
\begin{eqnarray}
\E^{-}\cdot\J^{+}&=& 
\BGs^{+}\Et{+}\cdot\Et{+}+(\BGs^{-})^{-1}\Jn{+}\cdot\Jn{+}.
\label{interface_formula2_equiv4}
\end{eqnarray}
Similar arguments give
\begin{eqnarray}
\E^{+}\cdot\J^{-}&=& 
\BGs^{-}\Et{+}\cdot\Et{+}+(\BGs^{+})^{-1}\Jn{+}\cdot\Jn{+}.
\label{interface_formula2_equiv5}
\end{eqnarray}
Taking the difference between (\ref{interface_formula2_equiv4}) and (\ref{interface_formula2_equiv5}) and substituting the result in 
(\ref{interface_formula2_equiv1}) gives (\ref{interface_formula2_equiv2}), completing the proof.
\end{proof}
\noindent {\em Remark.}
\begin{enumerate}
\item[$\bullet$] Since $\Et{-}=\Et{+}, \Jn{-}=\Jn{+}$, 
(\ref{interface_formula2_equiv2}) can equivalently be written as
\end{enumerate}
\begin{eqnarray*}
\langle \E\rangle\cdot\delta\BGs^*\langle \E\rangle&\approx&\frac{1}{|\Omega|}\int_{\Gamma} h \left[
(\BGs^+-\BGs^-) \Et{-}\cdot\Et{-}-\left((\BGs^+)^{-1}-(\BGs^-)^{-1}\right)\Jn{-}\cdot \Jn{-}\right].\\
\end{eqnarray*}
\subsection{Single interphase}
\label{sec:single}
\begin{figure}[t]
\centering
\scalebox{0.08}{\includegraphics{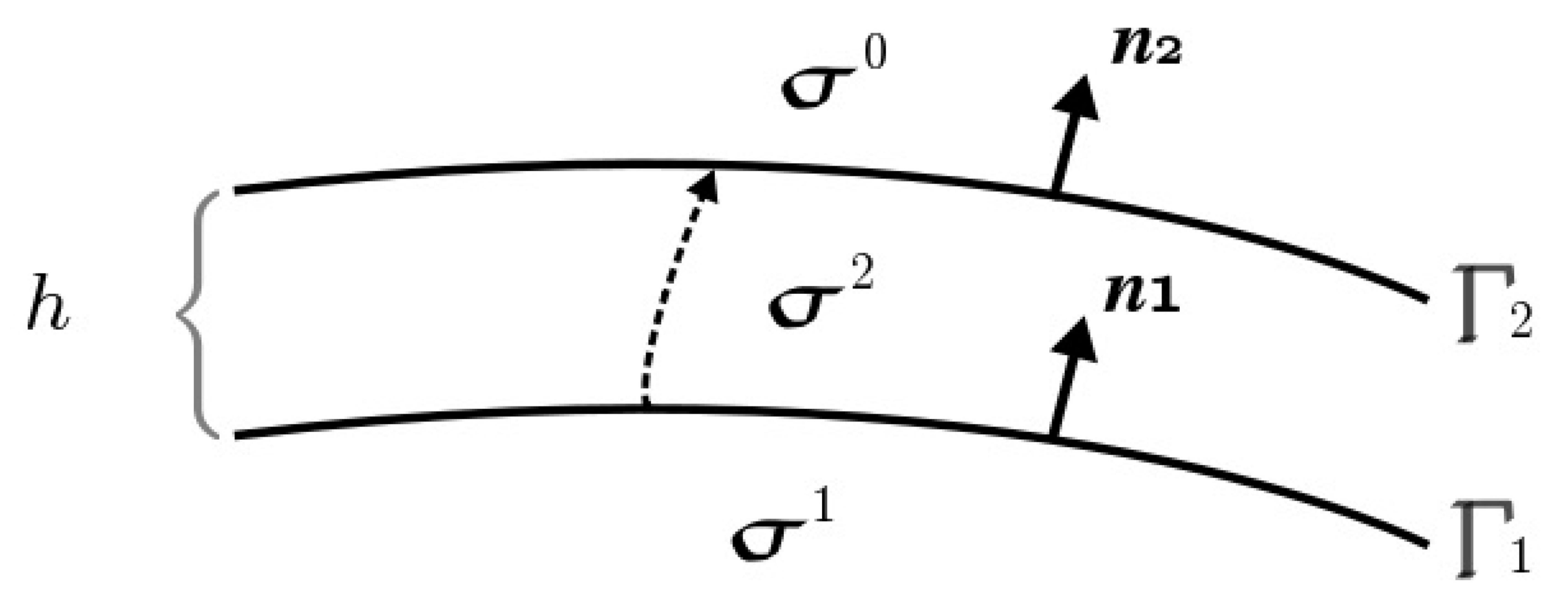}} 
\caption{A single-phase interphase region. In the limit as $\Gamma_1\rightarrow \Gamma_2$, 
the interphase vanishes and is replaced by the interface $\Gamma_2$ which now separates phase-$0$  and phase-$1$.
The figure is schematic in that the thickness $h$ is not assumed to be constant but 
can vary on a length scale large compared to $h$.}
\label{fig:single}
\end{figure}

Consider a composite in which two phases, with conductivities $\BGs^0$ and $\BGs^1$, are separated  by an interphase of 
conductivity $\BGs^2$ and thickness $h$, 
see Figure~\ref{fig:single}. We note that the thickness $h$ is not assumed to be constant but can vary slowly along the interphase region. 
Let $\Gamma_1$ and $\Gamma_2$ denote the interfaces between phase-$1$ and phase-$2$, and phase-$2$ and phase-$0$, respectively. 
We may assume, without loss of generality, that the directions of the normal vectors $\n_1$ and $\n_2$ are given 
according to Figure~\ref{fig:single}. 
We also assume that the interphase geometry is locally laminar, which in particular implies that $\n_1$ approaches $\n_2$, as $h$ approaches
zero.
Then by applying Lemma \ref{lem:1} we find that the change in the effective conductivity due to translating $\Gamma_1$ by a distance $h$ in the direction of $\n_1$ is
\begin{eqnarray}
\label{single_interphase1}
\nonumber
\langle \E\rangle\cdot\delta\BGs^*\langle \E\rangle&\approx&\frac{1}{|\Omega|}\int_{\Gamma_1} h 
\left[\E^{2}\cdot\J^{1}-\E^{1}\cdot\J^{2}\right]\\
\nonumber
&=&\frac{1}{|\Omega|}\int_{\Gamma_1} h \left[
(\BGs^1-\BGs^2) \Et{1}\cdot\Et{1}-\left((\BGs^1)^{-1}-(\BGs^2)^{-1}\right)\Jn{1}\cdot \Jn{1}\right].\\
\end{eqnarray}
\subsection{Two interphases}
\label{sec:double}
Consider a composite in which two phases, with conductivities $\BGs^0$ and $\BGs^1$, are separated  by two  interphases of 
conductivities $\BGs^2$ and $\BGs^3$, and  thickness $h_2 h$ and $h_3 h$, where $h_2+h_3=1$, see Figure~\ref{fig:two}. 
We note that the thicknesses $h$, $h_2 h$ and $h_3 h$ are not assumed to be constant but can vary slowly along the interphase region.
Let 
$\Gamma_1$, $\Gamma_2$, and $\Gamma_3$ denote the interfaces between phase-$1$ and phase-$2$, phase-$2$ and phase-$3$, and 
phase-$3$ and phase-$0$, respectively. We may assume, without loss of generality, that the directions of the normal 
vectors $\n_1$, $\n_2$ , and $\n_3$ are given according to Figure~\ref{fig:two}. We assume further that the 
interfaces $\Gamma_1$ and $\Gamma_2$ are translated simultaneously by distances $h=h_2 h+h_3 h$ and $h_3 h$ in the 
directions $\n_1$ and $\n_2$, respectively. 
We again assume that the interphase geometry is locally laminar, which implies that $\n_1$ and $\n_2$ approach $\n_3$.
Then we have the following result.
\begin{thm}
The change in the effective conductivity $\delta\BGs^*$ due to the change in the composite as described above is  given through the formula
\begin{eqnarray}
\label{double_interphase0}
\nonumber
\langle \E\rangle\cdot\delta\BGs^*\langle \E\rangle&\approx&
\frac{1}{|\Omega|}\int_{\Gamma_1} h \left[
(\BGs^1-(h_2\BGs^2+h_3\BGs^3)) \Et{1}\cdot\Et{1}\right.\\
& & -\left((\BGs^1)^{-1}-\left(h_2 (\BGs^2)^{-1}+ h_3 (\BGs^3)^{-1}\right)\right)\Jn{1}\cdot \Jn{1}].	
\end{eqnarray}
\end{thm}
\begin{proof}
By using Lemma \ref{lem:1} and (\ref{interface_formula3}) we obtain that 
the resultant change in the effective conductivity is  given by
\begin{eqnarray}
\label{double_interphase1}
\nonumber
\langle \E\rangle\cdot\delta\BGs^*\langle \E\rangle&\approx&\frac{1}{|\Omega|}\int_{\Gamma_1} h 
\left[\E^{2}\cdot\J^{1}-\E^{1}\cdot\J^{2}\right] + 
\frac{1}{|\Omega|}\int_{\Gamma_2} h h_3
\left[\E^{3}\cdot\J^{2}-\E^{2}\cdot\J^{3}\right]\\
\nonumber
&=&\frac{1}{|\Omega|}\int_{\Gamma_1} h\left[ 
(\BGs^1-\BGs^2) \Et{1}\cdot\Et{1}-\left((\BGs^1)^{-1}-(\BGs^2)^{-1}\right)\Jn{1}\cdot \Jn{1}\right]\\
\nonumber
& & +
\frac{1}{|\Omega|}\int_{\Gamma_1} h  h_3
(\BGs^2-\BGs^3) \Et{2}\cdot\Et{2}-\left((\BGs^2)^{-1}-(\BGs^3)^{-1}\right)\Jn{2}\cdot \Jn{2}.\\
\end{eqnarray}
\begin{figure}[t]
\centering
\scalebox{0.08}{\includegraphics{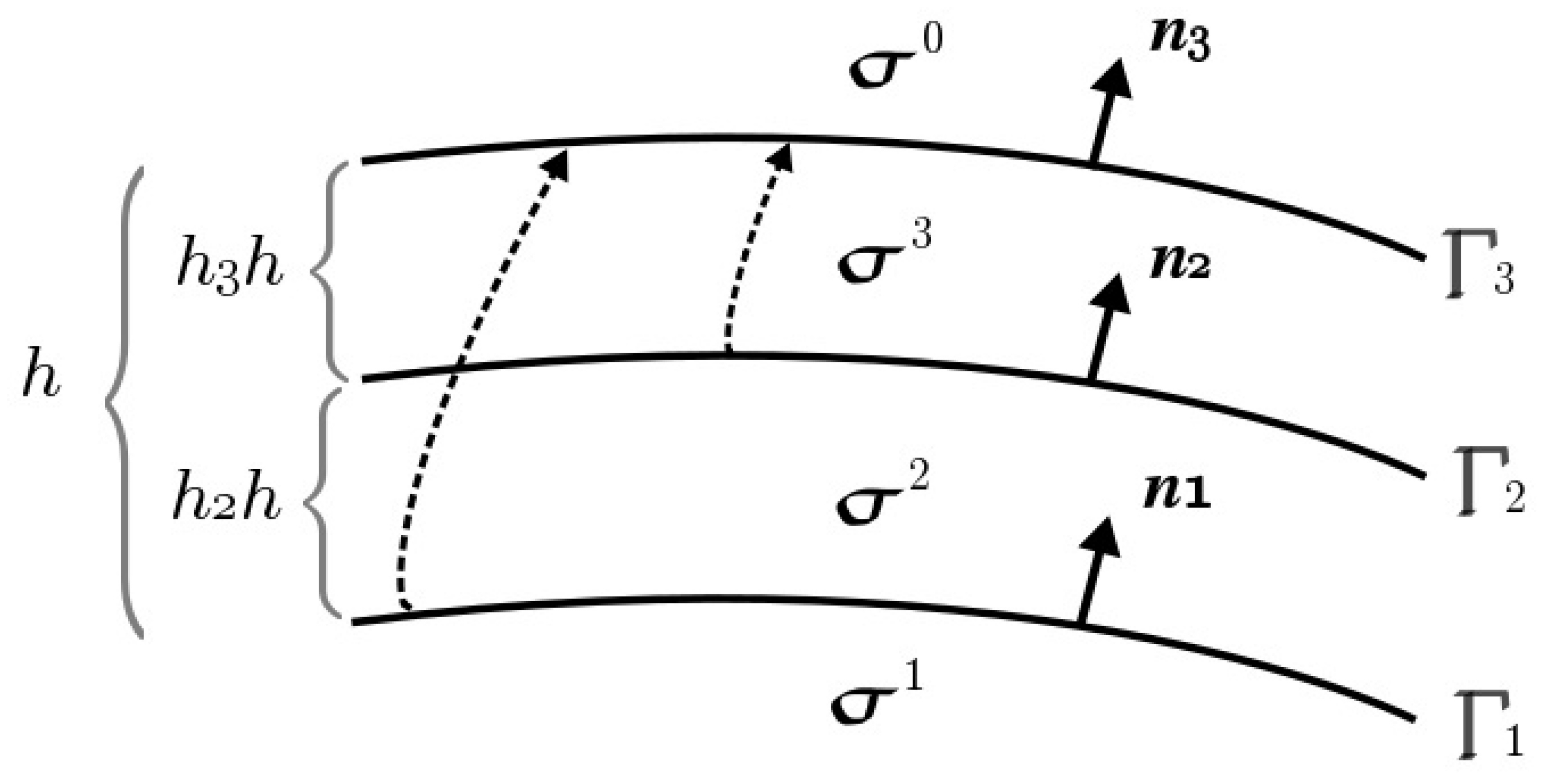}} 
\caption{A two-phase interphase region. In the limit as $\Gamma_1\rightarrow \Gamma_3$ and 
$\Gamma_2\rightarrow \Gamma_3$,  
the interphase region vanishes and is replaced by the interface $\Gamma_3$ which now separates phase-$0$  and phase-$1$.
The figure is schematic in that the thicknesses $h$, $h_2 h$  and $h_3 h$ are not assumed 
to be constant but can vary on a length scale large compared to these thicknesses.}
\label{fig:two}
\end{figure}
By taking  the limit as $h\rightarrow 0$, we have that $\Gamma_2 \approx \Gamma_1$, $\Et{2}\approx\Et{1}$, and $\Jn{2}\approx\Jn{1}$. 
Therefore, we can write (\ref{double_interphase1}) as
\begin{eqnarray}
\label{double_interphase2}
\nonumber
\langle \E\rangle\cdot\delta\BGs^*\langle \E\rangle&\approx&
\frac{1}{|\Omega|}\int_{\Gamma_1} h \left[
(\BGs^1-\BGs^2+h_3\BGs^2-h_3\BGs^3) \Et{1}\cdot\Et{1}\right.\\
\nonumber
& & - \left((\BGs^1)^{-1}-(\BGs^2)^{-1}+h_3 (\BGs^2)^{-1}-h_3 (\BGs^3)^{-1}\right)\Jn{1}\cdot \Jn{1}].\\
\end{eqnarray}
From (\ref{double_interphase2}) and by using the fact that $1-h_3=h_2$, equation (\ref{double_interphase0}) follows.
\end{proof}

\subsection{Multiple interphases}
\label{sec:multiple}
\begin{figure}[t]
\centering
\scalebox{0.10}{\includegraphics{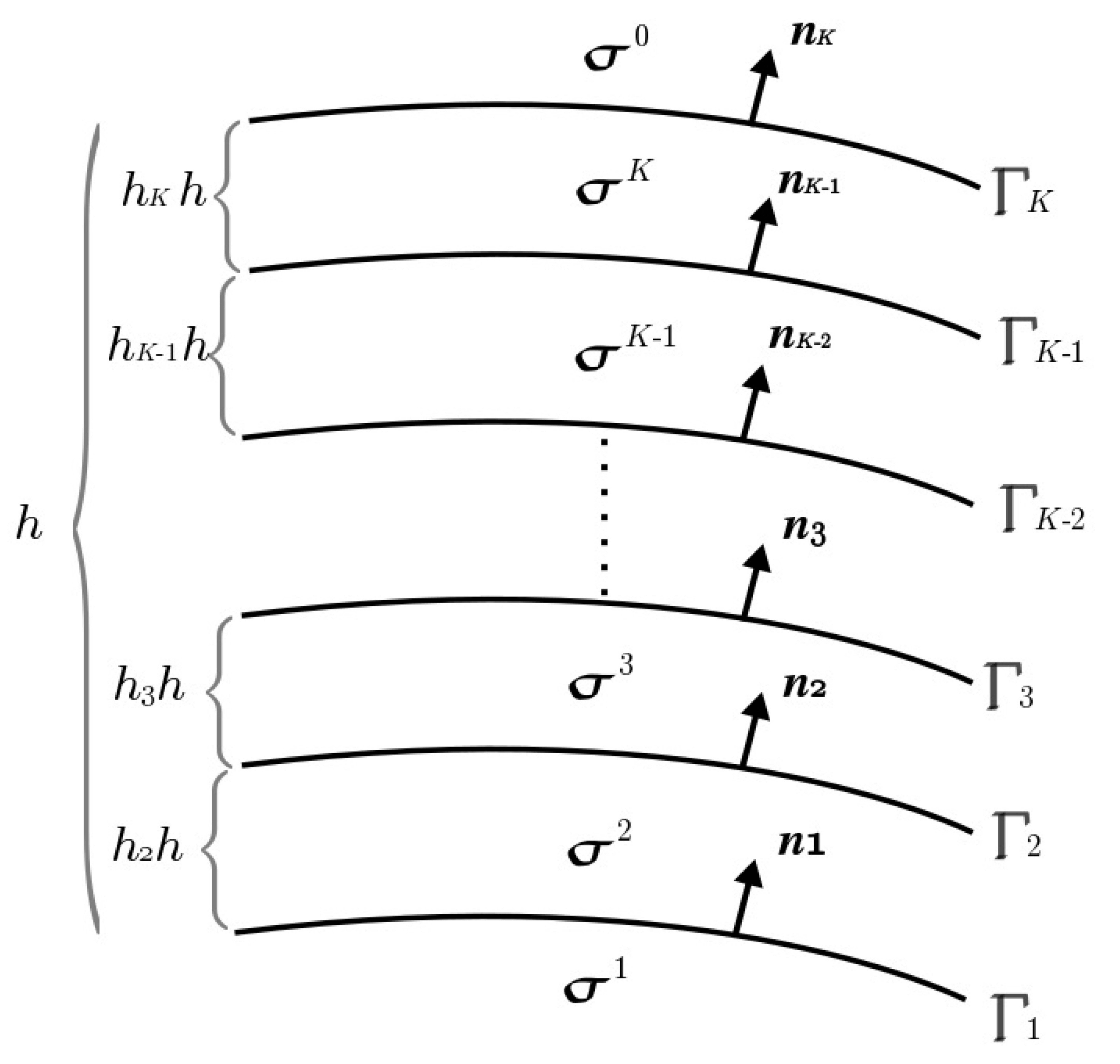}} 
\caption{A multi-phase interphase region. In the limit as $\Gamma_1\rightarrow \Gamma_K,$ 
 $\Gamma_2\rightarrow \Gamma_K,\ldots, \Gamma_{K-1}\rightarrow \Gamma_K$, 
the interphase region vanishes and is replaced by the interface $\Gamma_K$ which now separates phase-$0$  and phase-$1$.
The figure is schematic in that the thicknesses $h$, $h_2 h, h_3 h,\ldots,h_K h$ are not 
assumed to be constant but can vary on a length scale large compared to these thicknesses.}
\label{fig:multi}
\end{figure}
Consider a composite in which two phases, with conductivities $\BGs^0$ and $\BGs^1$, are separated  by $K$ interphases of conductivities 
$\BGs^2,\BGs^3,\ldots,\BGs^K$, and  thickness $h_2 h, h_3 h,\ldots,h_K$, 
where $\displaystyle\sum_{k=2}^{K}h_k=1$, see Figure~\ref{fig:multi}. 
We note that the thicknesses $h, h_2 h, h_3 h, \ldots, h_K h$ are not assumed to be constant but can vary slowly along the interphase region.
Let $\Gamma_1,\Gamma_2, \ldots, \Gamma_K$ denote 
the interfaces between phase-$1$ and phase-$2$, phase-$2$ and phase-$3$, $\dots$, and phase-$K$ and phase-$0$, respectively. We 
may assume, without loss of generality, that the directions of the normal vectors $\n_1,\n_2,\ldots,\n_K$ are given according 
to Figure~\ref{fig:multi}. We assume further that the interfaces $\Gamma_1,\ldots,\Gamma_{K-1}$ are translated simultaneously 
such that each $\Gamma_k$, $k=1,\ldots,K-1$, is translated by a distance $\displaystyle h\sum_{i=k+1}^{K}h_i$ in the 
direction $\n_k$ normal to the interface $\Gamma_k$. Then we have the following result.
\begin{thm}
The change in the effective conductivity $\delta\BGs^*$ is  given through the formula
\begin{eqnarray}
\label{graded_interphase1}
\nonumber
\langle \E\rangle\cdot\delta\BGs^*\langle \E\rangle&\approx&
\frac{1}{|\Omega|}\int_{\Gamma_1} h 
\left(\BGs^1-\sum_{k=2}^{K}h_k \BGs^k\right) \Et{1}\cdot\Et{1}\\
&& -h\left((\BGs^1)^{-1}-\sum_{k=2}^{K} h_k (\BGs^k)^{-1}\right)\Jn{1}\cdot \Jn{1}.
\end{eqnarray}
\end{thm}
\begin{proof}
From (\ref{interface_formula3}) we see that the resultant change in the energy is
\begin{eqnarray}
\label{graded_interphase3}
\nonumber
\langle \E\rangle\cdot\delta\BGs^*\langle \E\rangle&\approx&
\sum_{k=1}^{K-1}\frac{1}{|\Omega|}\int_{\Gamma_k} h \left(\sum_{i=k+1}^{K-1}h_{i}\right)
\left[\E^{k+1}\cdot\J^{k}-\E^{k}\cdot\J^{k+1}\right].\\
\end{eqnarray}
Then, by using Lemma \ref{lem:1}, (\ref{graded_interphase3}) can be rewritten as
\begin{eqnarray}
\label{graded_interphase2}
\nonumber
\langle \E\rangle\cdot\delta\BGs^*\langle \E\rangle&\approx&
\sum_{k=1}^{K-1}\frac{1}{|\Omega|}\int_{\Gamma_k} h [ \left(\sum_{i=k+1}^{K-1}h_{i}\right)
(\BGs^k-\BGs^{k+1}) \Et{k}\cdot\Et{k}\\
& & -\left((\BGs^k)^{-1}-(\BGs^{k+1})^{-1}\right)\Jn{k}\cdot \Jn{k}].
\end{eqnarray}
In the limit as $h\rightarrow 0$, we have that  $\Gamma_k \approx \Gamma_1$, $\Et{k}\approx\Et{1}$, and $\Jn{k}\approx\Jn{1}$, for $k=2,3,\ldots, K-1$. 
Therefore, we can rewrite (\ref{graded_interphase2}) as
\begin{eqnarray}
\label{graded_interphase4}
\nonumber
\langle \E\rangle\cdot\delta\BGs^*\langle \E\rangle&\approx&
\frac{1}{|\Omega|}\int_{\Gamma_1}h [ \sum_{k=1}^{K-1} \left(\sum_{i=k+1}^{K-1}h_{i}\right)
(\BGs^k-\BGs^{k+1}) \Et{1}\cdot\Et{1}\\
\nonumber
& & -h \sum_{k=1}^{K-1} \left(\sum_{i=k+1}^{K-1}h_{i}\right)\left((\BGs^k)^{-1}-(\BGs^{k+1})^{-1}\right)\Jn{1}\cdot \Jn{1}].\\
\end{eqnarray}
By expanding the sums and using the fact that $\displaystyle\sum_{k=2}^{K}h_k=1$, we find that
\begin{eqnarray}
\label{graded_sum1}
\sum_{k=1}^{K-1} \sum_{i=k+1}^{K-1}h_{i}
(\BGs^k-\BGs^{k+1})&=& \BGs^1 - \sum_{k=2}^{K-1}h_k\BGs^k,\\
\nonumber
\\
\label{graded_sum2}
\sum_{k=1}^{K-1} \sum_{i=k+1}^{K-1}h_{i}
\left((\BGs^k)^{-1}-(\BGs^{k+1})^{-1}\right) &=&
(\BGs^1)^{-1} - \sum_{k=2}^{K-1} h_k (\BGs^k)^{-1}.
\end{eqnarray}
The result follows from substituting (\ref{graded_sum1}) and (\ref{graded_sum2}) in (\ref{graded_interphase4}).

\end{proof}

\subsection{Graded interface}
\label{sec:graded}
By taking the limit of (\ref{graded_interphase1}) as $K\rightarrow\infty$, one obtains the following result for an interface with 
graded conductivity
\begin{eqnarray}
\label{graded_interface1}
\nonumber
\langle \E\rangle\cdot\delta\BGs^*\langle \E\rangle&\approx&
\frac{1}{|\Omega|}\int_{\Gamma_1} \left[h 
\left(\BGs^1-\frac{1}{h}\int_{0}^{h}\BGs(z)\;dz\right) \Et{1}\cdot\Et{1}\right.\\
&&\left. -h\left((\BGs^1)^{-1}-\frac{1}{h}\int_{0}^{h} (\BGs)^{-1}(z)\; dz\right)\Jn{1}\cdot \Jn{1}\right].
\end{eqnarray}
It is worth noting that this formula applies even if $\BGs$ has a finite jump at $z=0$ or $z=h$.

\begin{figure}[t]
\begin{center}
\scalebox{0.065}{\includegraphics{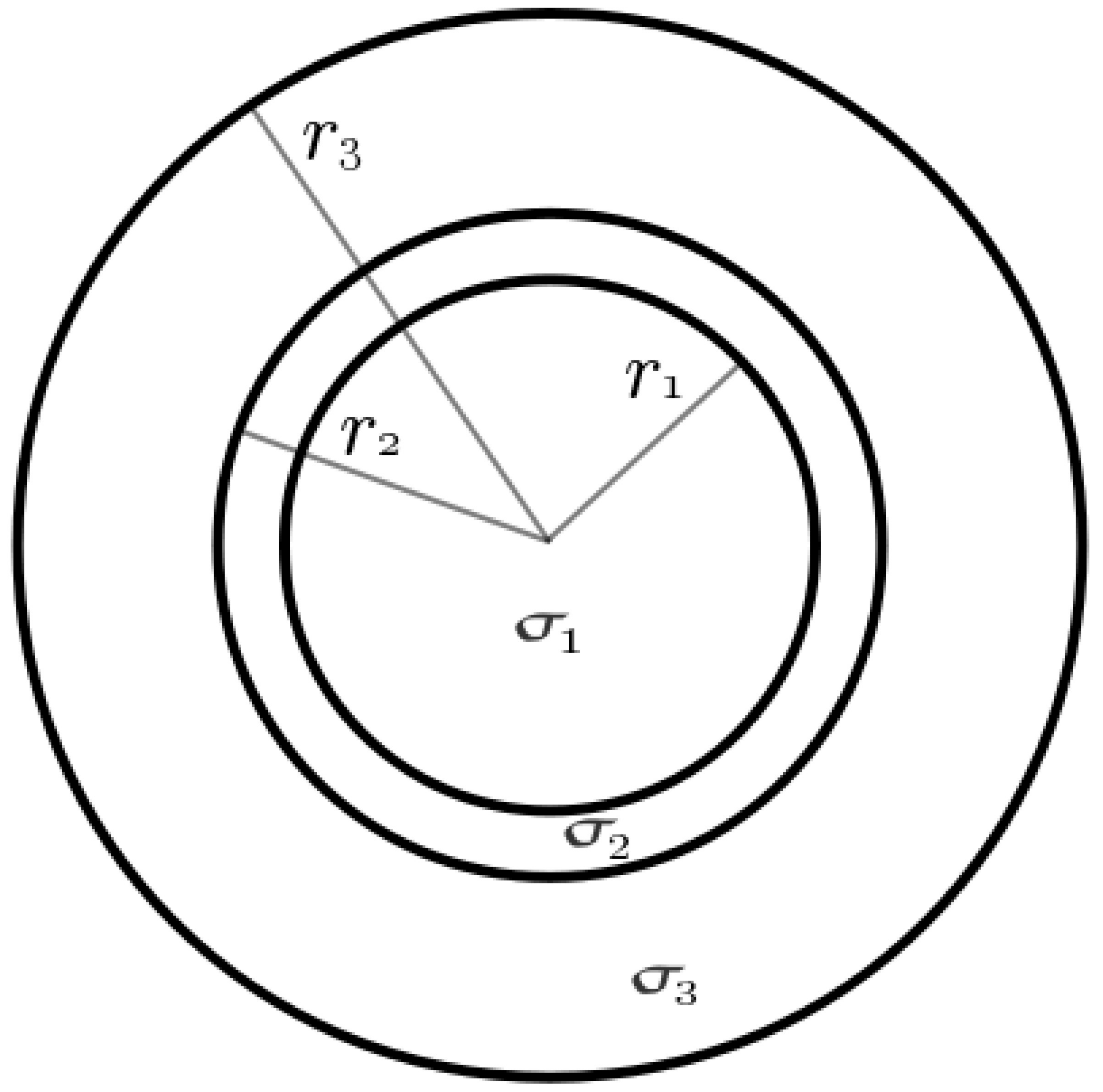}} 
\end{center}
\caption{A doubly coated sphere from the doubly coated sphere assemblage.}
\label{fig:coat_1}
\end{figure}

\section{Numerical Results}
\label{sec:numerical}
In this section, we consider a specific example of a composite with an interphase; namely the doubly coated sphere assemblage 
for a three-phase material, see for instance \cite{Milton:2002:TOC}. The reference composite is the well-known singly
coated sphere assemblage of Hashin and Shtrikman \cite{Hashin:1962:EMH, Hashin:1962:VAT}. 
The goal is to provide a numerical comparison between our method 
for approximating the effective 
conductivity and the exact effective conductivity of the doubly coated sphere assemblage. The numerical comparison also includes two other well-known
approximations for the effective conductivity for the cases of highly conducting interphase and poorly conducting interphase.

For convenience, we assume that the phases of the doubly coated sphere assemblage have isotropic conductivities and hence, for this
composite, the effective conductivity is also isotropic. Therefore, in this section, we represent conductivities by scalar quantities
and denote by $\sigma_*$, $\delta \sigma_*$, $\sigma_*^0$, and $\sigma_i$, for $i=1,2,3$, 
the effective conductivity, 
the change in the effective conductivity, the effective conductivity of the reference composite, 
and the conductivity of the $i$-th phase, respectively.

\begin{figure}[t]
\begin{center}

\hspace*{-1.3cm}\scalebox{0.267311}{\includegraphics{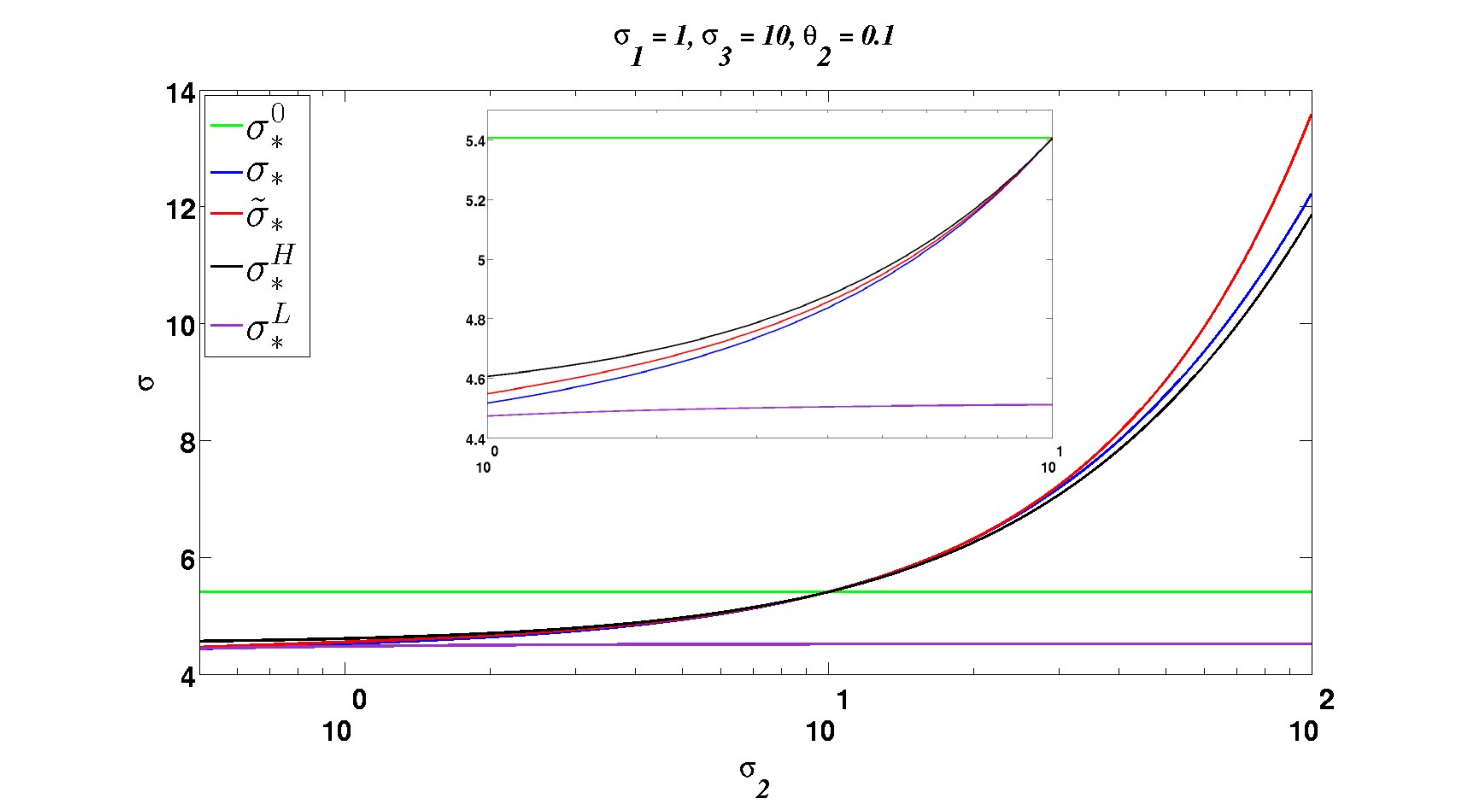}}
\end{center}
\caption{Log plot of the effective conductivity for the doubly coated sphere assemblage as a function of the interphase
conductivity $\sigma_2$. The interphase volume fraction is $\theta_2=0.1$. The core volume fraction is $ \theta_1=3^3/4^3\approx 0.422$, which
is the same in both the original and reference configurations.
The outer coating and the core have conductivities 
$\sigma_1=1,\sigma_3=10$, respectively. The interphase conductivity values are in the range $\sigma_2\in [10^{-2},10^{2}]$ 
with a zoomed-in view in the middle 
for conductivity values  $\sigma_2\in [1,10]$.}
\label{fig:num_1}
\end{figure}

\subsection{The doubly coated sphere assemblage}
\label{subsec:doubly}
Consider an arbitrary doubly coated sphere in the doubly coated sphere assemblage in which the three concentric spheres have radii
given by $r_1,r_2,r_3$ with $r_1<r_2<r_3$, see Figure \ref{fig:coat_1}. 
Assume that the core is occupied by a material with 
conductivity $\sigma_1 $, 
the interphase is occupied by a second material with isotropic conductivity $\sigma_2 $, and the outer-most coating 
is occupied by a third material with isotropic conductivity $\sigma_3 $. 
With this notation, the effective conductivity of the doubly 
coated sphere assemblage can be computed explicitly 
(see for example, \cite{Schulgasser:1977:CET} and Section 7.2 of \cite{Milton:2002:TOC}) and is given by
\begin{eqnarray}
\label{exact_effective}
 \sigma_*= \sigma_3+ \frac{3 \sigma_3 (1-\theta_3)}{\theta_3-\displaystyle\frac{3 \sigma_3}{\sigma_3-\sigma_2-\frac{\displaystyle 3 \sigma_2 \theta_1}{
\displaystyle\theta_2-
\frac{3 \sigma_2 (1-\theta_3)}{\sigma_2-\sigma_1}}}},
\end{eqnarray}
where 
\begin{eqnarray}
\label{thetas}
\nonumber
\theta_1 &=& \left(\frac{r_1}{r_3}\right)^3,\\
\theta_2 &=& \left(\frac{r_2}{r_3}\right)^3-\left(\frac{r_1}{r_3}\right)^3,\\
\nonumber
\theta_3 &=& 1-\left(\frac{r_2}{r_3}\right)^3,
\end{eqnarray}
are the volume fractions of phase-$1$, phase-$2$, and phase-$3$, respectively.

It is important to note that the thickness of the interphase, phase-$2$, varies from one doubly coated sphere to the next doubly coated sphere, 
in proportion to the relative sizes of the doubly coated spheres, so that the doubly coated spheres are rescaled copies of each other.

\begin{figure}[t]
\begin{center}

\hspace*{-1.3cm}\scalebox{0.267311}{\includegraphics{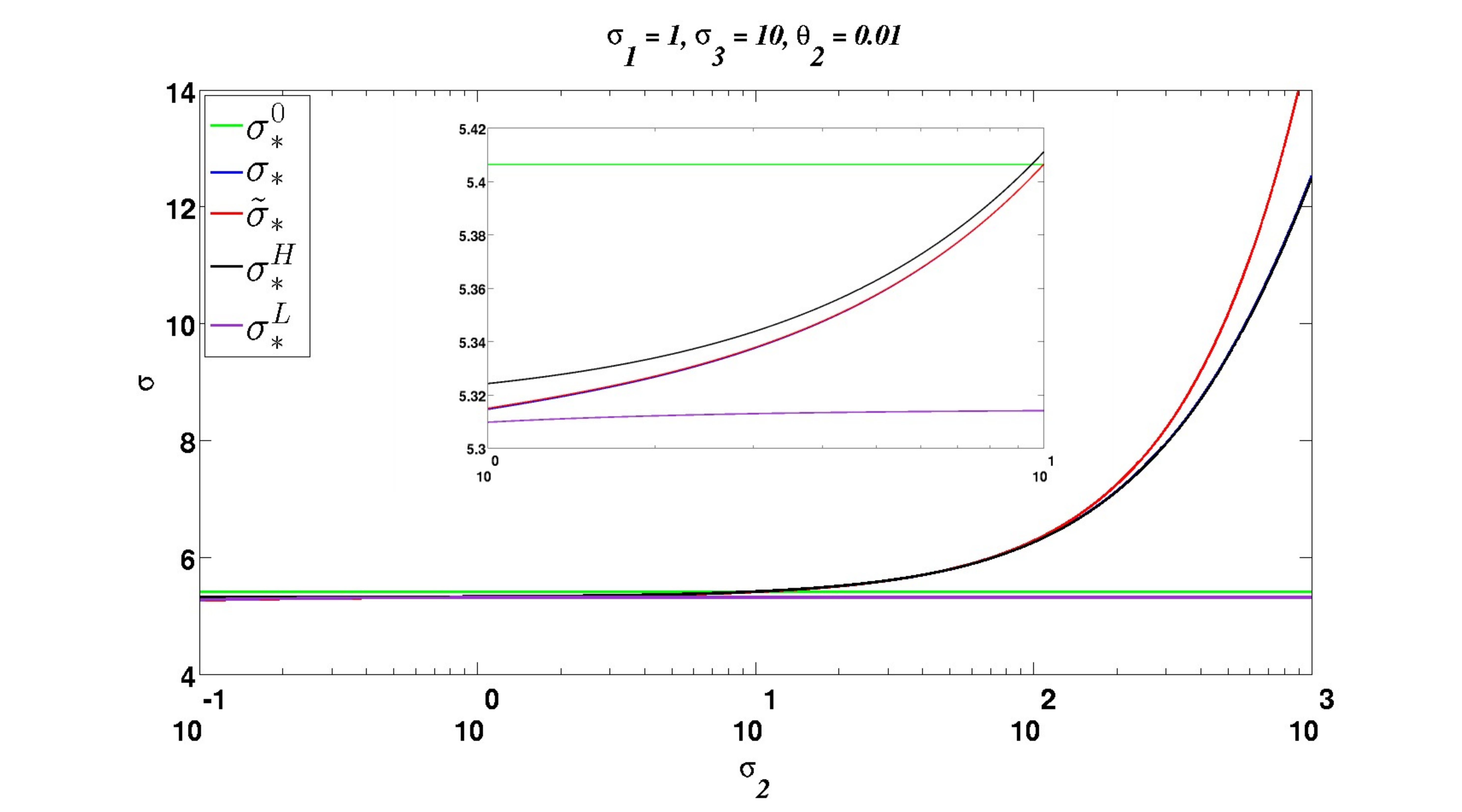}}
\end{center}
\caption{Log plot of the effective conductivity for the doubly coated sphere assemblage as a function of the interphase
conductivity $\sigma_2$. The interphase volume fraction is $\theta_2=0.01$.
The core volume fraction is $ \theta_1=3^3/4^3\approx 0.422$, which
is the same in both the original and reference configurations.
The outer coating and the core have conductivities 
$\sigma_1=1,\sigma_3=10$, respectively. The interphase conductivity values are in the range $\sigma_2\in [10^{-2},10^{3}]$ 
with a zoomed-in view in the middle 
for conductivity values  $\sigma_2\in [1,10]$.}
\label{fig:num_2}
\end{figure}

\begin{figure}[t]
\begin{center}

\hspace*{-1.3cm}\scalebox{0.267311}{\includegraphics{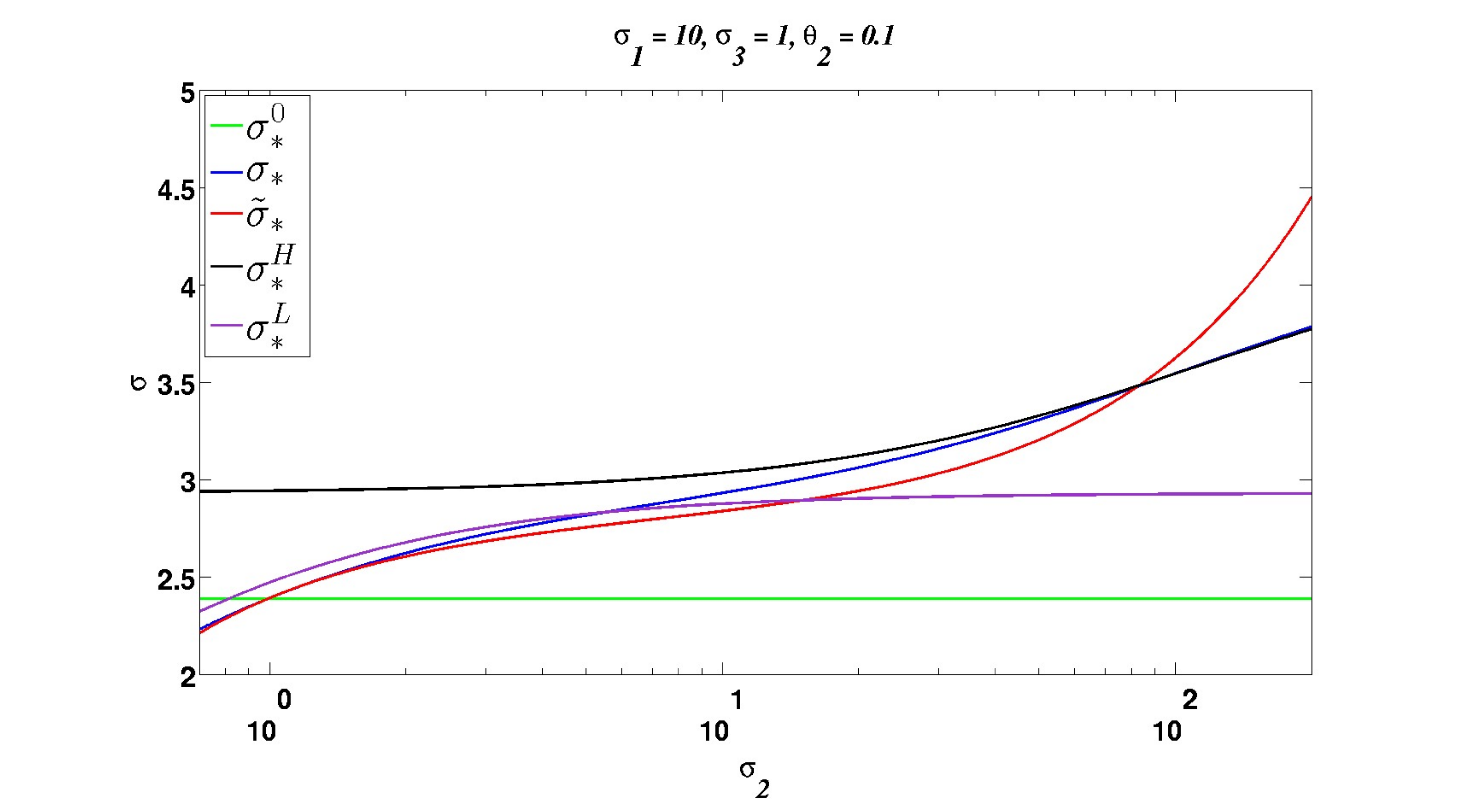}}
\end{center}
\caption{Log plot of the effective conductivity for the doubly coated sphere assemblage as a function of the interphase
conductivity $\sigma_2$. The interphase volume fraction is $\theta_2=0.1$.
The core volume fraction is $ \theta_1=3^3/4^3\approx 0.422$, which
is the same in both the original and reference configurations.
The outer coating and the core have conductivities 
$\sigma_1=10,\sigma_3=1$, respectively.}
\label{fig:num_3}
\end{figure}
\subsection{The approximate effective conductivity $\widetilde{\sigma}_*$}
In the doubly coated sphere assemblage described in the previous section, let $r_1$ be fixed and $r_2=(1+t)r_1$ for some $t>0$.
Then the thickness of the interphase is $h=r_2-r_1=t r_1$. By letting $t\rightarrow 0$, the thickness of the interphase 
$h\rightarrow 0$, and hence we can compute the change in the effective conductivity $\delta \sigma_*$ using (\ref{single_interphase1}).
However, in this case it is more convenient to compute $\delta \sigma_*$ from noticing that 
$\displaystyle\lim_{ h\rightarrow 0}\frac{\delta \sigma_*}{ h}=\left.\frac{d \sigma_*}{dh}\right|_{h=0}$ 
and that $\sigma_*$ is given explicitly by (\ref{exact_effective}). Thus, by using the fact that $h=r_2-r_1$, substituting  
(\ref{thetas}) in (\ref{exact_effective}), and then differentiating with respect to $h$, one obtains
\begin{eqnarray}
\label{delta_sigma}
\nonumber
 \delta \sigma_*(h) &\approx& h \left(\left.\frac{d \sigma_*}{dh}\right|_{h=0}\right)\\
&=& h \;\frac{-9 \sigma_3 \theta_1(\sigma_3-\sigma_2)((\sigma_1)^2+2 \sigma_2 \sigma_3)}{r_1 \sigma_2 ((\sigma_3-\sigma_1) \theta_1+\sigma_1+2 \sigma_3)^2}.
\end{eqnarray}
On the other hand, the effective conductivity of the reference configuration (the unperturbed problem) $\sigma_*^0$ can be calculated through
evaluating $\sigma_*$ at $h=0$, which after simplification gives 
\begin{eqnarray}
\label{sigma_unperturbed}
\nonumber
 \sigma_*^0&=&\sigma_*(0)\\
&=&\sigma_3+ \frac{3 \sigma_3 \theta_1 }{1-\theta_1-\displaystyle\frac{3 \sigma_3}{\sigma_3-\sigma_1}}.
\end{eqnarray}
By substituting (\ref{delta_sigma}) and (\ref{sigma_unperturbed}) in the scalar version of (\ref{segma_effctive_appx}), 
one obtains the following formula
for our approximation of the effective conductivity for this composite 
\begin{eqnarray}
\label{sigma_appx}
\nonumber
 \sigma_*(h)&\approx&\widetilde{\sigma}_*(h)\\
\nonumber
&=& \sigma_3+ \frac{3 \sigma_3 \theta_1 }{1-\theta_1-\displaystyle\frac{3 \sigma_3}{\sigma_3-\sigma_1}}
+h \;\frac{-9 \sigma_3 \theta_1(\sigma_3-\sigma_2)((\sigma_1)^2+2 \sigma_2 \sigma_3)}{r_1 \sigma_2 ((\sigma_3-\sigma_1) \theta_1+\sigma_1+2 \sigma_3)^2}.\\
\end{eqnarray}

\begin{figure}[t]
\begin{center}

\hspace*{-1.3cm}\scalebox{0.267311}{\includegraphics{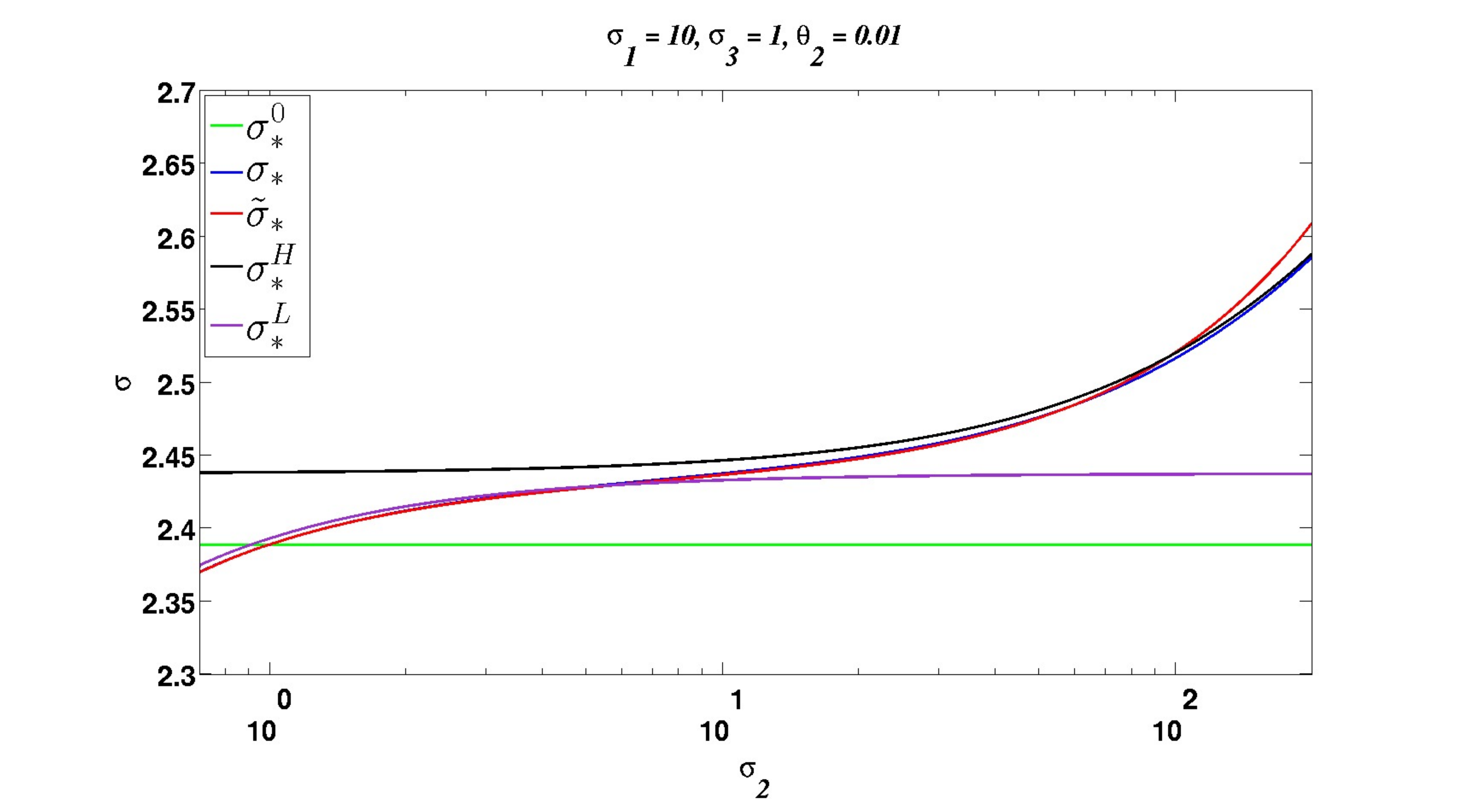}}

\end{center}
\caption{Log plot of the effective conductivity for the doubly coated sphere assemblage as a function of the interphase
conductivity $\sigma_2$. The interphase volume fraction is $\theta_2=0.01$.
The core volume fraction is $ \theta_1=3^3/4^3\approx 0.422$, which
is the same in both the original and reference configurations.
The outer coating and the core have conductivities 
$\sigma_1=10,\sigma_3=1$, respectively.}
\label{fig:num_4}
\end{figure}
\subsection{Approximations of the effective conductivity for high-contrast cases}
We consider the doubly coated sphere assemblage described in Section \ref{subsec:doubly} for the cases in which the interphase is thin
and is either highly conducting or poorly conducting. The high-conduction case corresponds to assuming that 
$\theta_2\rightarrow 0$ and $\sigma_2\rightarrow \infty$ in such a way that the product $\theta_2\sigma_2$ remains constant.
By taking these limits in (\ref{exact_effective}), we obtain that the effective 
conductivity in this case is given by
\begin{eqnarray}
\label{sigma_H}
 \sigma_*^H =\sigma_3+\frac{3 (1-\theta_3)\sigma_3}{\theta_3 -\displaystyle\frac{3 \sigma_3}{\sigma_3 - 
\sigma_1-\displaystyle\frac{2 \theta_2 \sigma_2}{3(1-\theta_3)}}}.
\end{eqnarray}
While the poor-conduction case corresponds to assuming that 
$\theta_2\rightarrow 0$ and $\sigma_2\rightarrow 0$ in such a way that the ratio $\sigma_2/\theta_2$ remains constant.
By taking these limits in (\ref{exact_effective}), we find that the effective 
conductivity in this case is given by
\begin{eqnarray}
\label{sigma_L}
 \sigma_*^L =\sigma_3+\frac{3 (1-\theta_3)\sigma_3}{\theta_3 -\displaystyle\frac{3 \sigma_3}{\sigma_3 - 
\displaystyle\frac{3}{\displaystyle\frac{3}{\sigma_1}+\displaystyle\frac{\theta_2}{(1-\theta_3)\sigma_2}}}}.
\end{eqnarray}
\subsection{Numerical comparison of the results}
A numerical comparison given by log plots for the effective conductivity of the doubly coated sphere assemblage as a function of the interphase
conductivity $\sigma_2$ is provided in Figures~\ref{fig:num_1}-\ref{fig:num_6}. The five curves displayed in each figure correspond to
the exact effective conductivity $\sigma_*$, given by (\ref{exact_effective}), our approximation $\widetilde{\sigma}^*$, given by
(\ref{sigma_appx}), the highly conducting interphase approximation $\sigma_H$, given by (\ref{sigma_H}), 
the poorly conducting interphase approximation $\sigma_L$, given by (\ref{sigma_L}), and the effective conductivity of the 
reference composite (the zero order approximation) $\sigma_*^0$, given by (\ref{sigma_unperturbed}). For a representative 
doubly coated sphere we have taken the outer radius $r_3=4$ length units and the inner radius $r_1=3$ length units and thus,
by using (\ref{thetas}),
the interphase thickness $h=r_2-r_1$ can be computed from the interphase volume fraction $\theta_2$, which we use as a parameter 
in generating Figures~\ref{fig:num_1}-\ref{fig:num_6}. For example, in Figure~\ref{fig:num_1}, a value of $\theta_2=0.1$ corresponds
to $h= 0.2204$, and a value of $\theta_2=0.01$, as in Figure~\ref{fig:num_2}, corresponds to $h= 0.0235$.

Figures~\ref{fig:num_1}-\ref{fig:num_6} show that there is a good agreement between the approximate effective conductivity
$\widetilde{\sigma}_*$ and the exact effective conductivity $\sigma_*$ for the doubly coated sphere assemblage. 
This agreement   is most noticeable when $\theta_2$, and consequently $h$, is small, as in Figures \ref{fig:num_2} and \ref{fig:num_4}, and
for values of $\sigma_2$ within or near the range $[\min \{\sigma_1,\sigma_3\}, \max \{\sigma_1,\sigma_3\}]$. This confirms that
$\widetilde{\sigma}_*$ well approximates $\sigma_*$ when the interphase is thin and has intermediate conductivity values.

\begin{figure}[t]
\begin{center}

\hspace*{-1.3cm}\scalebox{0.267311}{\includegraphics{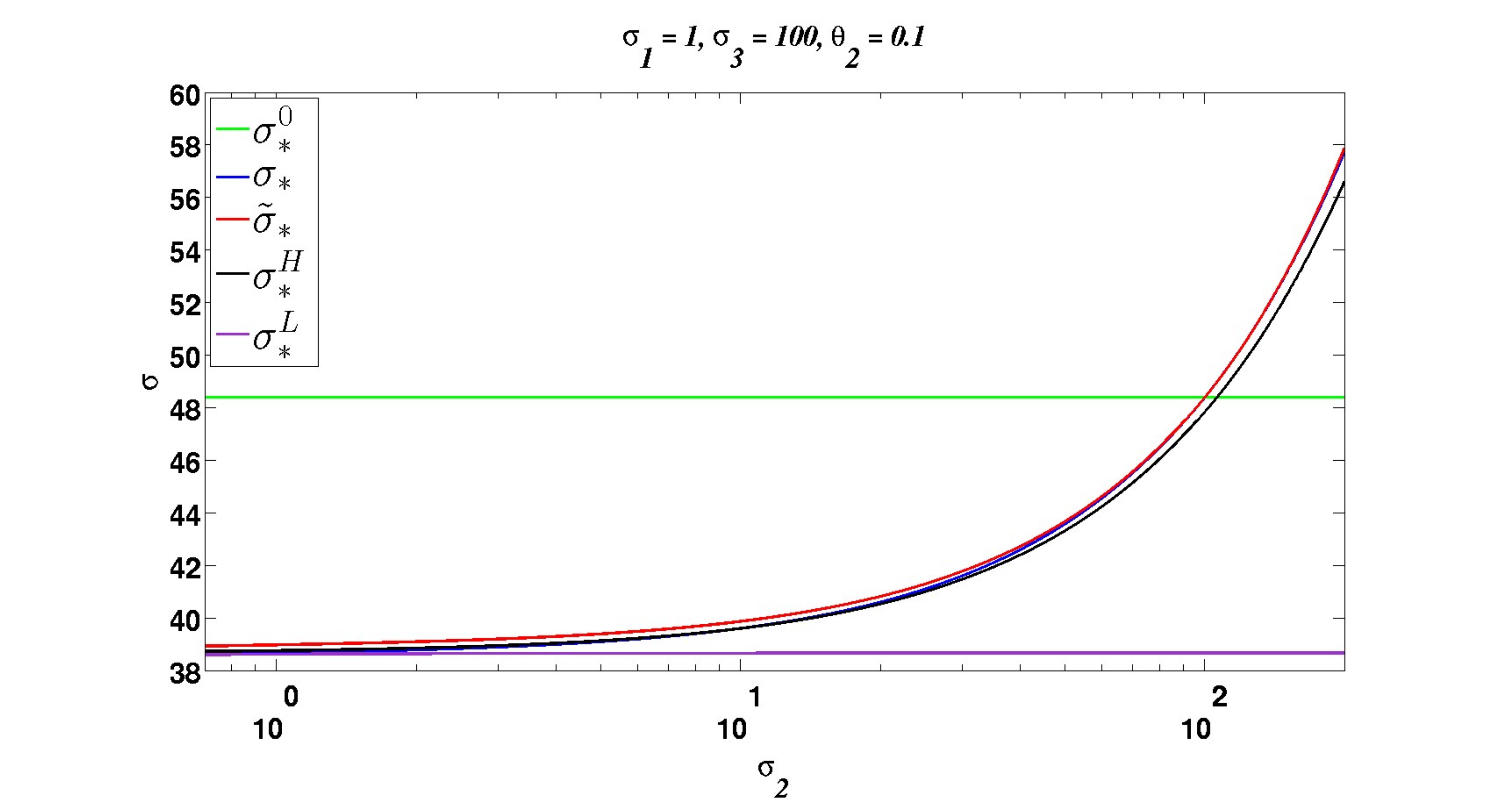}}
\end{center}
\caption{Log plot of the effective conductivity for the doubly coated sphere assemblage as a function of the interphase
conductivity $\sigma_2$. The interphase volume fraction is $\theta_2=0.1$.
The core volume fraction is $ \theta_1=3^3/4^3\approx 0.422$, which
is the same in both the original and reference configurations.
The outer coating and the core have conductivities 
$\sigma_1=1,\sigma_3=100$, respectively.}
\label{fig:num_5}
\end{figure}

\begin{figure}[t]
\begin{center}

\hspace*{-1.3cm}\scalebox{0.267311}{\includegraphics{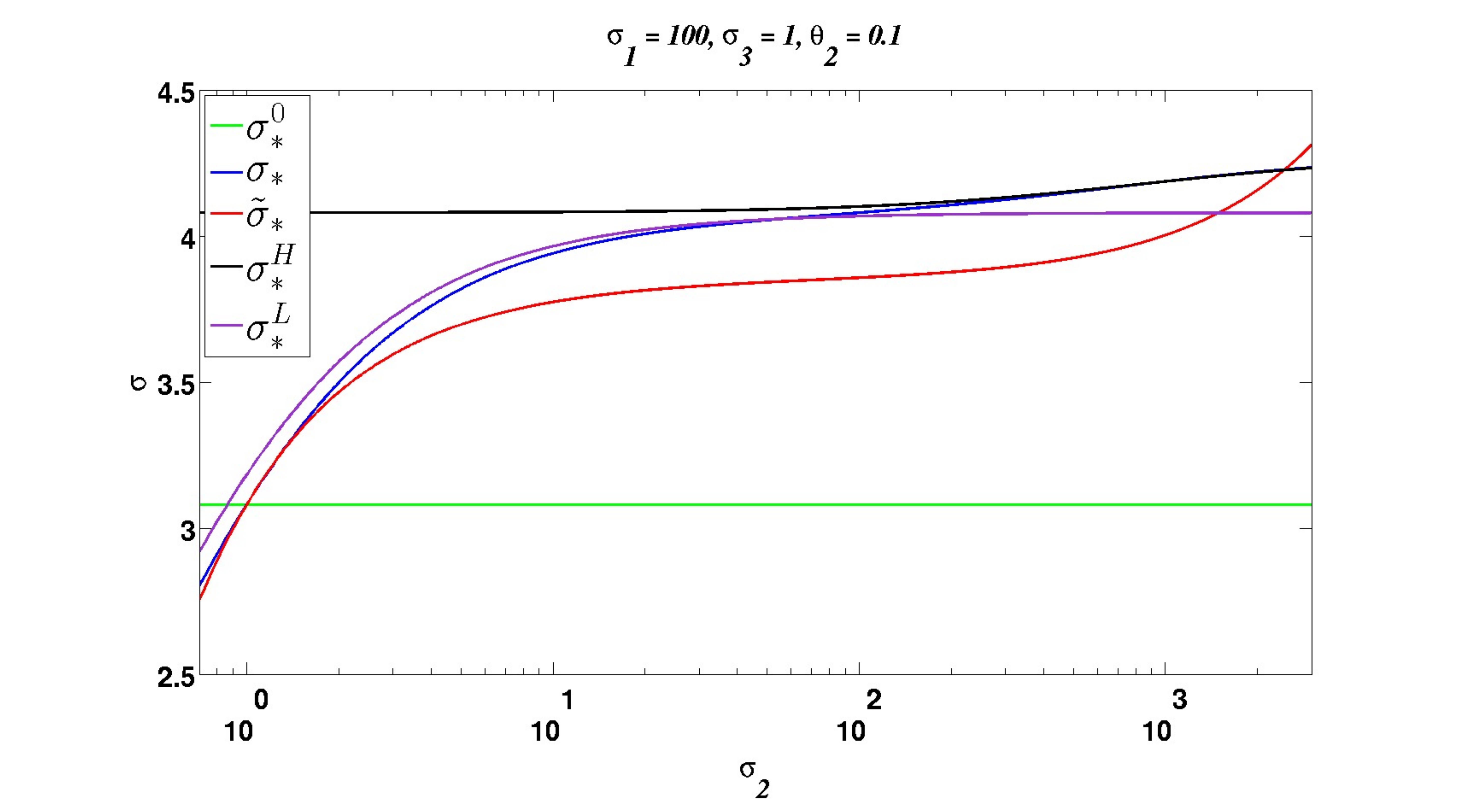}}
\end{center}
\caption{Log plot of the effective conductivity for the doubly coated sphere assemblage as a function of the interphase
conductivity $\sigma_2$. The interphase volume fraction is $\theta_2=0.1$.
The core volume fraction is $ \theta_1=3^3/4^3\approx 0.422$, which
is the same in both the original and reference configurations.
The outer coating and the core have conductivities 
$\sigma_1=100,\sigma_3=1$, respectively.}
\label{fig:num_6}
\end{figure}

\section*{Acknowledgment}
The authors would like to thank Yakov Benveniste for his valuable comments on this work.
The authors are grateful for support from the National Science Foundation through grants DMS-0707978 and DMS-1211359.

\bibliographystyle{chicago}
\bibliography{interphase}
%

\end{document}